\documentclass[sort&compress]{elsarticle}

\usepackage{setspace}
\usepackage{amsmath}
\usepackage{amssymb}
\usepackage{mathrsfs}
\usepackage{multirow}
\usepackage{color}
\usepackage{longtable}
\usepackage{array}
\usepackage{url}
\usepackage{comment}
\usepackage{enumerate}
\usepackage{eucal}
\usepackage{graphics}

\usepackage{algorithm}
\usepackage[noend]{algorithmic}
\usepackage[table]{xcolor}

\usepackage{stmaryrd}

\usepackage{mathtools}
\usepackage{xparse}
\DeclarePairedDelimiterX{\set}[1]{\{}{\}}{\setargs{#1}}
\NewDocumentCommand{\setargs}{>{\SplitArgument{1}{;}}m}
{\setargsaux#1}
\NewDocumentCommand{\setargsaux}{mm}
{\IfNoValueTF{#2}{#1} {#1\,\delimsize|\,\mathopen{}#2}}

\DeclarePairedDelimiter\abs{\lvert}{\rvert}
\DeclarePairedDelimiter\ceil{\lceil}{\rceil}
\DeclarePairedDelimiter\floor{\lfloor}{\rfloor}
\DeclarePairedDelimiter\parenv{\lparen}{\rparen}
\DeclarePairedDelimiter\sparenv{\lbrack}{\rbrack}

\newcommand{\cB}{\mathcal{B}}

\newcommand{\cG}{\mathcal{G}}

\newcommand{\bai}{\bar{i}}
\newcommand{\baj}{\bar{j}}

\renewcommand{\leq}{\leqslant}

\renewcommand{\geq}{\geqslant}

\newcommand{\ppmod}[1]{~({\rm mod~}#1)}

\newdefinition{definition}{Definition}
\newdefinition{remark}{Remark}
\newtheorem{theorem}{Theorem}
\newtheorem{example}{Example}
\newtheorem{corollary}{Corollary}
\newtheorem{lemma}{Lemma}
\newproof{proof}{Proof}

\newcommand{\F}{\mathbb{F}}
\newcommand{\R}{\mathbb{R}}
\newcommand{\Z}{\mathbb{Z}}

\newcommand{\ve}{\mathbf{e}}

\newcommand{\vr}{\mathbf{r}}

\newcommand{\vv}{\mathbf{v}}
\newcommand{\vu}{\mathbf{u}}
\newcommand{\vx}{\mathbf{x}}
\newcommand{\vy}{\mathbf{y}}
\newcommand{\vz}{\mathbf{z}}
\newcommand{\va}{\mathbf{a}}
\newcommand{\vb}{\mathbf{b}}
\newcommand{\vc}{\mathbf{c}}
\newcommand{\vd}{\mathbf{d}}
\newcommand{\vf}{\mathbf{f}}
\newcommand{\Zero}{{\mathbf{0}}}
\newcommand{\One}{{\mathbf{1}}}
\DeclareMathOperator{\wt}{wt}

\newcommand{\kp}{k_+}
\newcommand{\km}{k_-}
\newcommand{\BALL}{{\mathcal B}(n,t,\kp,\km)}
\newcommand{\eqdef}{\triangleq}
\newcommand{\splt}{\diamond}

\begin{document}

\title{On Tilings of Asymmetric Limited-Magnitude Balls\tnoteref{t1}}
\tnotetext[t1]{This work was supported in part by an Israel Science Foundation (ISF) grant 270/18.}
\author[1]{Hengjia Wei\corref{cor1}}
\ead{hjwei05@gmail.com}

\author[2]{Moshe Schwartz}
\ead{schwartz@ee.bgu.ac.il}

\cortext[cor1]{Corresponding author}
\address[1]{School of Electrical and Computer Engineering,
  Ben-Gurion University of the Negev,\\
  Beer Sheva 8410501, Israel}
\address[2]{
  School of Electrical and Computer Engineering,
  Ben-Gurion University of the Negev,\\
  Beer Sheva 8410501, Israel}

\begin{abstract}
  We study whether an asymmetric limited-magnitude ball may tile
  $\Z^n$. This ball generalizes previously studied shapes: crosses,
  semi-crosses, and quasi-crosses. Such tilings act as perfect
  error-correcting codes in a channel which changes a transmitted
  integer vector in a bounded number of entries by limited-magnitude
  errors.

  A construction of lattice tilings based on perfect codes in the
  Hamming metric is given. Several non-existence results are proved,
  both for general tilings, and lattice tilings. A complete
  classification of lattice tilings for two certain cases is proved.
\end{abstract}

\maketitle

\section{Introduction}

In some applications, information is encoded as a vector of integers,
$\vx\in\Z^n$, most notably, flash memories (e.g., see
\cite{CasSchBohBru10}). Additionally, a common noise affecting these
applications is a limited-magnitude error affecting some of the
entries. Namely, at most $t$ entries are increased by as much as $\kp$
or decreased by as much as $\km$. Thus, for integers $n\geq t\geq 1$, and
$\kp\geq \km\geq 0$, we define the \emph{$(n,t,\kp,\km)$-error-ball}
as
\[\BALL\triangleq \set*{\vx=(x_1,x_2,\ldots,x_n)\in \Z^n ; -\km \leq x_i \leq \kp \text{ and  } \wt(\vx) \leq t  },\]
where $\wt(\vx)$ denotes the Hamming weight of $\vx$. It now follows
that an error-correcting code in this setting is equivalent to a
packing of $\Z^n$ by $\BALL$, and the subject of interest for this
paper, a perfect code is equivalent to a tiling of $\Z^n$ by $\BALL$.
An example of $\cB(3,2,2,1)$ is shown in Fig.~\ref{fig:qc}.

\begin{figure}
  \begin{center}
    \includegraphics{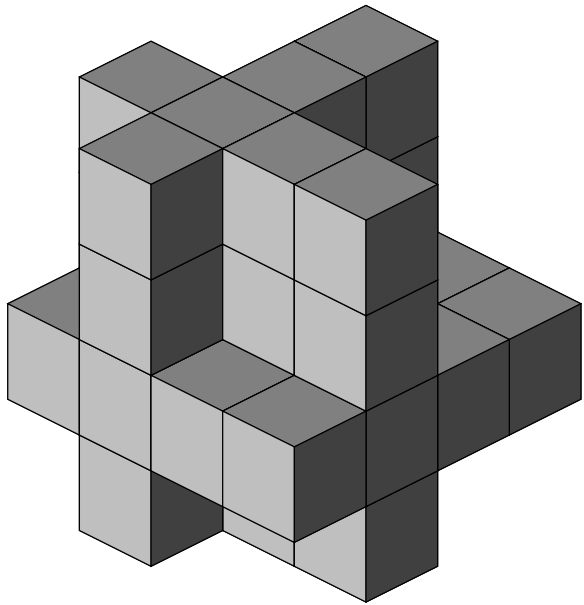}
  \end{center}
  \caption{A depiction of $\cB(3,2,2,1)$ where each point in
    $\cB(3,2,2,1)$ is shown as a unit cube.}
  \label{fig:qc}
\end{figure}

Previous works on tiling these shapes almost exclusively studied the
case of $t=1$. The \emph{cross}, $\cB(n,1,k,k)$, and semi-cross,
$\cB(n,1,k,0)$ have been extensively researched, e.g., see
\cite{Ste84,HamSte84,HicSte86,SteSza94,KloLuoNayYar11} and the many
references therein. This was recently extended to
\emph{quasi-crosses}, $\cB(n,1,\kp,\km)$, in \cite{Sch12}, creating a
flurry of activity on the subject
\cite{YarKloBos13,Sch14,ZhaGe16,ZhaZhaGe17,ZhaGe18,YeZhaZhaGe20}. To
the best of our knowledge, \cite{Ste90} and later \cite{BuzEtz12}, are
the only works to consider $t\geq 2$, by considering a notched cube
(or a ``chair''), which for certain parameters becomes
$\cB(n,n-1,k,0)$. Tilings of these shapes have been constructed in
\cite{Ste90,BuzEtz12}. Additionally, \cite{BuzEtz12} showed that
$\cB(n,n-2,k,0)$, $n\geq 4$, $k\geq 1$, can never lattice-tile $\Z^n$.

The goal of this paper is to study tilings of $\BALL$ for $t\geq
2$. Our main contributions are a construction of lattice tilings from
perfect codes in the Hamming metric, and a sequence of non-existence
results, both for lattice tilings and for general non-lattice
tilings. We use both algebraic techniques and geometric ones. In
particular, we provide a complete classification of lattice tilings
with $\cB(n,2,1,0)$ and $\cB(n,2,2,0)$.

The paper is organized as follows: In Section~\ref{sec:prelim} we
provide the notation used throughout the paper, as well as definitions
and basic results concerning lattice tilings and group splittings. We
construct lattice tilings in Section~\ref{sec:tiling}, and prove
non-existence results in Section~\ref{sec:notiling}. A short
discussion and open questions are given in Section~\ref{sec:conclude}.

\section{Preliminaries}
\label{sec:prelim}

Throughout the paper we let $n$ and $t$ be integers such that $n\geq
t\geq 1$. We further assume $\kp$ and $\km$ are non-negative integers
such that $\kp\geq \km\geq 0$. For integers $a\leq b$ we define
$[a,b]\eqdef\set*{a,a+1,\dots,b}$ and $[a,b]^*\eqdef
[a,b]\setminus\set*{0}$. We use $\Z_m$ to denote the cyclic group of
integers with addition modulo $m$, and $\F_q$ to denote the finite
field of size $q$. Since we shall almost always use just the additive
group of the finite field, when $p$ is a prime we shall sometimes
write $\F_p$ and sometimes $\Z_p$.

A \emph{lattice} $\Lambda\subseteq\Z^n$ is an additive subgroup of
$\Z^n$. A lattice $\Lambda$ may be represented by a matrix
$\cG(\Lambda)\in\Z^{n\times n}$, the span of whose rows (with integer
coefficients) is $\Lambda$. A \emph{fundamental region} of $\Lambda$
is defined as
\[ \set*{ \sum_{i=1}^n c_i\vv_i ; c_i\in\R, 0\leq c_i < 1 },\]
where $\vv_i$ is the $i$-th row of $\cG(\Lambda)$. It is well known
that the volume of the fundamental region is $\abs{\det(\cG(\Lambda))}$, and
is independent of the choice of $\cG(\Lambda)$.

We say $\cB\subseteq\Z^n$ \emph{packs} $\Z^n$ by $\Lambda\subseteq\Z^n$, if
the translates of $\cB$ by elements from $\Lambda$ do not intersect,
namely, for all $\vv,\vv'\in\Lambda$, $\vv\neq\vv'$,
\[ (\vv+\cB)\cap(\vv'+\cB)=\varnothing.\]
We say $\cB$ \emph{covers} $\Z^n$ by $\Lambda$ if
\[ \bigcup_{\vv\in\Lambda} (\vv+\cB) = \Z^n.\]
If $\cB$ both packs and covers $\Z^n$ by $\Lambda$, then we say $\cB$
\emph{tiles} $\Z^n$ by $\Lambda$. It is well known that if $\cB$ packs
$\Z^n$ by $\Lambda$, and $\abs*{\cB}=\abs{\det(\cG(\Lambda))}$, then $\cB$
tiles $\Z^n$ by $\Lambda$.

\subsection{Lattice Tiling and Group Splitting}

Lattice tiling of $\Z^n$ with $\cB(n,t,\kp,\km)$, in connection with
group splitting, has a long history when $t=1$ (e.g., see
\cite{Ste67}), called lattice tiling by crosses if $\kp=\km$ (e.g.,
\cite{Ste84}), semi-crosses when $\km=0$ (e.g.,
\cite{Ste84,HamSte84,HicSte86}), and quasi-crosses when
$\kp\geq\km\geq 0$ (e.g., \cite{Sch12,Sch14}). For an excellent
treatment and history, the reader is referred to~\cite{SteSza94} and
the many references therein. Other variations, keeping $t=1$
include~\cite{Tam98,Tam05}. More recent results may be found
in~\cite{YeZhaZhaGe20} and the references therein.

Since we are interested in codes that correct more than one error,
namely, $t\geq 2$, an extended definition of group splitting is
required.

\begin{definition}
  \label{def:split}
  Let $G$ be a finite Abelian group, where $+$ denotes the group
  operation. For $m\in\Z$ and $g\in G$, let $mg$ denote $g+g+\dots+g$
  (with $m$ copies of $g$) when $m>0$, which is extended in the
  natural way to $m\leq 0$. Let $M\subseteq\Z\setminus\set*{0}$ be a
  finite set, and $S=\set*{s_1,s_2,\dots,s_n}\subseteq G$. We say the
  set $M$ $t$-splits $G$ with splitter set $S$, denoted
  \[G = M\splt_t S\]
  if the following two conditions hold:
  \begin{enumerate}
  \item
    The elements $\ve\cdot (s_1,\dots,s_n)$, where $\ve\in
    (M\cup\set{0})^n$ and $1\leq \wt(\ve)\leq t$, are all distinct and
    non-zero in $G$.
  \item
    For every $g\in G$ there exists a vector $\ve\in
    (M\cup\set{0})^n$, $\wt(\ve)\leq t$, such that $g=\ve\cdot
    (s_1,\dots,s_n)$.
  \end{enumerate}
\end{definition}

Intuitively, $G=M\splt_t S$ means that the non-trivial linear
combinations of elements from $S$, with at most $t$ non-zero
coefficients from $M$, are distinct and give all the non-zero elements
of $G$ exactly once. We note that when $t=1$, this definition
coincides with the definition of splitting used in previous papers.

The following two theorems show the equivalence of $t$-splittings and
lattice tilings, summarizing Lemma 3, Lemma 4, and Corollary 1
in \cite{BuzEtz12}. They generalize the treatment for $t=1$ in previous
works (e.g., see \cite{SteSza94}).

\begin{theorem}[Lemma 4 and Corollary 1 in \cite{BuzEtz12}]
  \label{th:lattotile}
  Let $G$ be a finite Abelian group, $M\eqdef [-\km,\kp]^*$, and
  $S=\set*{s_1,\dots,s_n}\subseteq G$, such that $G = M\splt_t
  S$. Define $\phi:\Z^n\to G$ as
  $\phi(\vx)\eqdef\vx\cdot(s_1,\dots,s_n)$ and let
  $\Lambda\eqdef\ker\phi$ be a lattice. Then $\cB(n,t,\kp,\km)$ tiles
  $\Z^n$ by $\Lambda$.
\end{theorem}


\begin{theorem}[Lemma 3 and Corollary 1 in \cite{BuzEtz12}]
  \label{th:tiletolat}
  Let $\Lambda\subseteq\Z^n$ be a lattice, and assume
  $\cB(n,t,\kp,\km)$ tiles $\Z^n$ by $\Lambda$. Then there exists a
  finite Abelian group $G$ and $S=\set*{s_1,s_2,\dots,s_n}\subseteq G$
  such that $G=M\splt_t S$, where $M\eqdef [-\km,\kp]^*$.
\end{theorem}

\section{Construction of Lattice Tilings}
\label{sec:tiling}

In this section we describe a construction for tilings with
$\BALL$. The method described here takes a linear perfect code in the
well known and extensively studied Hamming metric, and uses it to
construct the tiling. The obvious downside to this method is the fact
that very few perfect codes exist in the Hamming metric (see
\cite{MacSlo78} for more on perfect codes).

\begin{theorem}
  \label{th:perfectcode}
  In the Hamming metric space, let $C$ be a perfect linear
  $[n,k,2t+1]$ code over $\F_p$, with $p$ a prime. If $\kp+\km+1=p$,
  then
  \[ \Lambda\eqdef \set*{ \vx\in\Z^n ; (\vx\bmod p)\in C}\]
  is a lattice, and $\cB(n,t,\kp,\km)$ lattice-tiles $\Z^n$ by
  $\Lambda$.
\end{theorem}
\begin{proof}
  Directly from its definition, $\Lambda$ is closed under addition and
  under multiplication by integers. Thus, $\Lambda$ is a
  lattice. Denote $\cB\eqdef\cB(n,t,\kp,\km)$, and we now prove $\cB$
  tiles $\Z^n$ by $\Lambda$.

  To show packing, assume $\vv+\ve=\vv'+\ve'$, for some
  $\vv,\vv'\in\Lambda$ and $\ve,\ve'\in\cB$. But then
  $\ve-\ve'=\vv'-\vv\in\Lambda$, and by the definition of $\Lambda$,
  also $\ve''\eqdef ((\ve-\ve')\bmod p)\in C$. We note that
  $\wt(\ve)\leq t$ and $\wt(\ve')\leq t$, hence $\wt(\ve'')\leq
  2t$. By the minimum distance of $C$ this implies that
  $\ve''=\Zero$. Now, since each entry of $\ve-\ve'$ is in the range
  $[-(\kp+\km),\kp+\km]$, and since $\kp+\km+1=p$, we necessarily have
  that $\ve-\ve'=\Zero$, which in turn implies $\vv-\vv'=\Zero$. It
  follows that translates of $\cB$ by $\Lambda$ pack $\Z^n$.

  To show covering, let $\vx\in\Z^n$ be any integer vector. Then
  $\vx'\eqdef(\vx\bmod p)\in\F_p^n$. Since $C$ is a perfect code,
  there exists $\vv'\in C$ and $\ve'\in\F_p^n$, $\wt(\ve')\leq t$,
  such that $\vx'\equiv \vv'+\ve'\pmod{p}$. Since $\kp+\km+1=p$, there
  exists $\ve\in\cB$ such that $\ve\bmod p=\ve'$. But then
  $\vx-\ve\equiv \vv'\pmod{p}$ and by definition
  $\vx-\ve\in\Lambda$. Hence, the translates of $\cB$ by $\Lambda$
  cover $\Z^n$.  \qed\end{proof}

\begin{example}
  Take the $[\frac{p^m-1}{p-1},\frac{p^m-1}{p-1}-m,3]$ $p$-ary Hamming
  code ($p$ a prime), together with Theorem~\ref{th:perfectcode}, to
  obtain a tiling of $\Z^{(p^m-1)/(p-1)}$ by
  $\cB(\frac{p^m-1}{p-1},1,\kp,\km)$, where $\kp+\km+1=p$. This
  particular tiling was already described in~\cite{Sch12} together
  with the lattice generator matrix and equivalent splitting.
\end{example}

\begin{example}
  \label{ex:ex2}
  If we use Theorem~\ref{th:perfectcode} with the perfect binary
  linear $[2t+1,1,2t+1]$ repetition code, we obtain a lattice tiling
  of $\Z^{2t+1}$ by $\cB(2t+1,t,1,0)$. The lattice is spanned by
  \[ \cG=\begin{pmatrix}
    1 & 1 & 1 & \dots  & 1 \\
      & 2 &   &        &   \\
      &   & 2 &        &   \\
      &   &   & \ddots &   \\
      &   &   &        & 2
  \end{pmatrix}.\]
  When viewed as a splitting, the additive group $\F_2^{2t}$ is
  $t$-split as $\F_2^{2t}=\set{1}\splt_t S$, where $S=\set{\ve_i ;
    1\leq i\leq 2t}\cup\set{\One}$, and where $\ve_i$ is the $i$-th
  unit vector of length $2t$.
\end{example}

\begin{example}
  Again using Theorem~\ref{th:perfectcode} with the $[23,12,7]$ binary
  Golay code, we obtain a lattice tiling of $\Z^{23}$ by
  $\cB(23,3,1,0)$. The lattice $\Lambda$ is spanned by
  \[\cG=\begin{pmatrix}
  I_{12} & G_b \\
  \Zero & 2 I_{11}
  \end{pmatrix},\]
  where $\begin{pmatrix} I_{12} & G_b \end{pmatrix}$ is a generator
  matrix of the $[23,12,7]$ binary Golay code, and $2I_{11}$ is an
  $11\times 11$ matrix with entries on the diagonal being $2$ and all
  the others being $0$. Now, we look at the corresponding group
  splitting. Since $\Z^{23}$ can be spanned by the matrix
  \[\begin{pmatrix}
  I_{12} & G_b \\
  \Zero & I_{11}
  \end{pmatrix},\]
  the quotient group $\Z^{23} / \Lambda$ is isomorphic to the additive
  group $\F_2^{11}$. Note that
  \[ \begin{pmatrix}
    I_{12} & G_b \\
    \Zero & 2 I_{11}
  \end{pmatrix}
  \begin{pmatrix}
    G_b \\
    I_{11}
  \end{pmatrix}
  \]
  is a $23 \times 11$ all-zero matrix over $\F_2$.  The natural
  homomorphism $\phi: \Z^{23} \to \F_2^{11}$ sends the standard basis
  to the rows of $\begin{pmatrix} G_b \\ I_{11} \end{pmatrix}$. It
  follows that $\F_2^{11}=\set{1}\splt_3 S$, where $S=\set{\ve_i ;
    1\leq i\leq 11}\cup\set{\vr; \vr \text{ is a row of }G_b }$.
\end{example}

\begin{example}\label{ex:ex4}
  Finally, using Theorem~\ref{th:perfectcode} with the $[11,6,5]$
  ternary Golay code, we obtain a lattice tiling of $\Z^{11}$ by
  $\cB(11,2,2,0)$ or $\cB(11,2,1,1)$. The lattice is spanned by 
  \[\cG=\begin{pmatrix}
  I_6 & G_t \\
  \Zero & 3 I_5
  \end{pmatrix},\]
  where $\begin{pmatrix} I_6 & G_t \end{pmatrix}$ is a generator
  matrix of the $[11,6,5]$ ternary Golay code, and $3I_5$ is a
  $5\times 5$ matrix with entries on the diagonal being $3$ and all
  the others being $0$. When viewed as a splitting, the additive group
  $\F_3^{5}$ is $2$-split as $\F_3^{5}=\set{1,2}\splt_2 S$, where
  $S=\set{\ve_i ; 1\leq i\leq 5}\cup\set{\vr; \vr \textup{ is a row of
    }G_t }$.
\end{example}

Theorem~\ref{th:perfectcode} has its dual as well, as shown in the
following theorem.

\begin{theorem}
  \label{th:revperfectcode}
  Assume $\BALL$ lattice-tiles $\Z^n$ by the lattice
  $\Lambda$, with an equivalent $t$-splitting $\F_p^m= M\splt_t S$,
  where $M\eqdef[-\km,\kp]^*$, $p$ is a prime, and $p=\kp+\km+1$. Then
  $\Lambda\cap\F_p^n$ is a perfect linear $[n,k,2t+1]$ code over
  $\F_p$ in the Hamming metric space.
\end{theorem}
\begin{proof}
  By Theorem~\ref{th:lattotile} and Theorem~\ref{th:tiletolat},
  $\Lambda=\ker\phi$, where $\phi:\Z^n\to\F_p^m$, with
  $S=\set{s_1,\dots,s_n}\subseteq\F_p^m$, and
  $\phi(\vx)=\vx\cdot(s_1,\dots,s_n)$. Let $\ve_i\in\Z^n$ be the
  $i$-th standard unit vector. Due to the characteristic of $\F_p^n$,
  for all $\vx\in\Z^n$, $\phi(\vx)=\phi(\vx+p\ve_i)$. It follows that
  \begin{equation}
    \label{eq:movelambda}
    \Lambda=\Lambda+p\ve_i,
  \end{equation}
  for all $i=1,2,\dots,n$. In turn, this implies that
  \begin{equation}
    \label{eq:lambdamod}
    \Lambda\cap\F_p^n = \Lambda\bmod p \eqdef \set*{\vx\bmod p ; \vx\in\Lambda}.
  \end{equation}
  Since $\Lambda$ is a lattice, we then have that $C\eqdef\Lambda\cap\F_p^n$
  is a vector space, namely, a linear code.

  It remains to show $C$ is a perfect code with the claimed
  parameters. Let $\vc,\vc'\in C$ be two distinct codewords, and
  $\ve,\ve'\in\F_p^n$ be two error patterns, $\wt(\ve),\wt(\ve')\leq
  t$. Assume to the contrary that
  \[\vc+\ve\equiv \vc'+\ve' \pmod{p},\]
  where we emphasize that addition here is in $\F_p^n$ by writing that
  the equivalence holds modulo $p$. Since $\kp+\km+1=p$, there are
  unique vectors $\vf,\vf'\in\BALL$ such that
  \[ \vf\equiv \ve\pmod{p} \qquad\text{and}\qquad \vf'\equiv\ve'\pmod{p}.\]
  We now have
  \[\vc+\vf\equiv \vc'+\vf' \pmod{p},\]
  hence there exists $\vv\in\Z^n$ such that
  \[\vc+\vf=\vc'+\vf'+p\vv.\]
  If we define $\vc''=\vc'+p\vv$, then by \eqref{eq:movelambda},
  $\vc''\in\Lambda$. But then
  \[ \vc+\vf=\vc''+\vf',\]
  contradicting the fact that $\BALL$ tiles $\Z^n$ by $\Lambda$. Thus,
  $C$ is a linear $[n,k,\geq 2t+1]$ code over $\F_p$.

  Finally, we show $C$ is perfect. Let $\vu\in\F_p^n$ be any
  vector. Since $\BALL$ tiles $\Z^n$ by $\Lambda$, there exist
  $\vv\in\Lambda$ and $\ve\in\BALL$ such that $\vu=\vv+\ve$. Taking
  the equation modulo $p$, we get that
  \[ \vu \equiv \vv+\ve \pmod{p},\]
  where we emphasize that $\vu\bmod p=\vu$. By~\eqref{eq:lambdamod},
  $\vv\bmod p\in C$. Additionally, since $\kp+\km+1=p$, we have that
  $\wt(\ve)=\wt(\ve\bmod p)\leq t$. Thus $C$ has covering radius at
  most $t$, and it is therefore a perfect code, as claimed.
  \qed\end{proof}

\section{Nonexistence Results}
\label{sec:notiling}

The nonexistence results we present in this section are divided into
results on general tilings, and results on lattice tilings. The former
use mainly geometric arguments, whereas the latter employ algebraic
ones.

\subsection{Nonexistence of General Tilings}

The first result we present uses a comparison between the density of a
tiling of $\BALL$ with that of a tiling of a certain notched cube of a
lower dimension.

\begin{theorem}
  \label{th:non1}
  For any $n\geq t+1$, and $\kp\geq \km\geq 0$ not both $0$, if
  \[\sum_{i=0}^t \binom{n}{i} (\kp+\km)^i <(\kp+1)^{t+1}-(\kp-\km)^{t+1}\]
  then $\Z^n$ cannot be tiled by translates of $\BALL$.
\end{theorem}

\begin{proof}
Given integers $n\geq t+1$, assume that there is a set $T \subseteq
\Z^n$ such that $\cB\eqdef\cB(n,t,\kp,\km)$ tiles $\Z^n$ by $T$. Consider the
set
\[
A=\set*{(x_1,x_2,\ldots,x_{t+1},  0, \ldots, 0) ; (x_1,\dots,x_{t+1})\in [0,\kp]^{t+1}\setminus [\km+1,\kp]^{t+1} }. 
\]
Hence, if we remove the last $n-t-1$ zero coordinates, the elements of
$A$ are exactly a notched cube, as defined in
\cite{Ste90,BuzEtz12}. Thus, by \cite{Ste90,BuzEtz12}, translates of
$A$ tile the space\footnote{While \cite{Ste90,BuzEtz12} discuss a
  tiling of $\R^n$, it is easily seen that the tiling constructed
  there is in fact a tiling of $\Z^n$ as in our setting.}
\[\set*{(x_1,x_2,\ldots,x_{t+1},  0, \ldots, 0) ;  x_i  \in \Z   \text{\ for all }  1\leq i \leq t+1 } .\]
Trivially, it follows that translates of $A$ can tile the space
$\Z^n$.

We now claim that any translate of $A$ contains at most one point from
$T$.  Suppose to the contrary that both $\vx=(x_1,x_2,\ldots,x_n)$ and
$\vy=(y_1,y_2, \ldots, y_n)$ belong to the intersection $(\vv+A) \cap
T$, where $\vv=(v_1,v_2,\ldots,v_n)\in\Z^n$, and $\vx\neq\vy$. Then
$v_i\leq x_i,y_i \leq v_i + \kp$ for $1\leq i \leq t+1$, $x_i=y_i=v_i$
for $t+2 \leq i \leq n$, and there are indices $1\leq j_x, j_y \leq
t+1$ such that $x_{j_x} \leq v_{j_x} +\km$ and $y_{j_y} \leq
v_{j_y}+\km$. W.l.o.g., assume that $x_1\leq v_1+\km$. We proceed in
two cases.
\begin{enumerate}
\item If $y_1\leq v_1+\km$, let $\vz=(z_1,z_2,\ldots,z_{t+1},v_{t+2},v_{t+3}\ldots,v_n)$, where
\begin{equation*}
\begin{split}
  z_1 = & \begin{cases}
    x_1, & \text{if $x_i\leq y_i$ for all $i=2,3,\ldots,t+1$,} \\
    y_1, & \text{otherwise,}
\end{cases}
\end{split}
\end{equation*}
and
\[z_i=\max\set{x_i,y_i} \text{\ for } i=2,3,\ldots, t+1.\] 
Then it is easy to see that
\[ \vz\in (\vx+\cB)\cap (\vy+\cB),\]
a contradiction.
\item If $y_1> v_1+ \km$, then there is $2\leq j \leq t+1$ such that
  $y_j\leq v_j+\km$. W.l.o.g., assume that $y_2\leq v_2+ \km$ and let
  $\vz=(y_1,z_2,z_3,\ldots,z_{t+1},v_{t+2},v_{t+3}\ldots,v_n)$, where
\begin{equation*}
\begin{split}
  z_2 = & \begin{cases}
    x_2, & \text{if $x_i\leq y_i$ for all $i=2,3,\ldots,t+1$,} \\
    \max\set{x_2,y_2}, &\text{otherwise,}
\end{cases}
\end{split}
\end{equation*}
and
\[z_i=\max\set{x_i,y_i} \text{\ for } i=3,4,\ldots, t+1.\] 
Again,
\[ \vz\in (\vx+\cB)\cap (\vy+\cB),\]
a contradiction.
\end{enumerate}

We have shown that any translate of $A$ contains at most one point
from $T$, and so the tiling by $A$ is denser than the tiling by $\cB$.
It follows that the reciprocal of the volume of $\cB$ cannot exceed
the reciprocal of the volume of $A$, i.e.,
\[\frac{1}{\sum_{i=0}^t \binom{n}{i} (\kp+\km)^i}\leq \frac{1}{(\kp+1)^{t+1}-(\kp-\km)^{t+1}}.\]
Rearranging gives us the desired result. 
\qed\end{proof}

\begin{remark}
  If  $\km\geq c\kp$ for some real number $c>0$, while $n$ and $t$ are fixed, then according
  to Theorem~\ref{th:non1}, there is an upper bound on $\kp$ for which
  $\BALL$ can tile $\Z^n$.
\end{remark}

  Next, we study a case which is analogous to that of proper
  quasi-crosses when $t=1$, namely, the case when $\kp>\km>0$. The
  main tool is a geometric one, studying the two translates of $\BALL$
  that cover the all-zero and all-one vectors.

\begin{theorem}\label{thm:nge2t-1}
Let $2t\geq n\geq t+1$ and $\kp > \km > 0$. Then $\Z^n$ cannot be
tiled by $\BALL$.
\end{theorem}

\begin{proof}
  Denote $\cB \triangleq \BALL$, and assume to the contrary that there
  is a set $T\subseteq \Z^n$ such that $\cB$ tiles $\Z^n$ by
  $T$. W.l.o.g., we may assume that the all-zero vector $\Zero$ is in
  $T$.

  We consider the all-one vector $\One$. Since $\One \not \in \cB$,
  there is a non-zero vector $\va=(a_1,a_2,\ldots, a_n) \in T$ such
  that $\One\in\va+\cB$, where $1-\kp\leq a_i \leq 1+\km$ for $1\leq
  i\leq n$.  By interchanging the coordinates, we may assume,
  w.l.o.g., that
  \[a_i=1 \text{\ for }1\leq i\leq n-t, \text{\ and\ } a_i \geq a_{i+1} \text{ for } n-t+1 \leq i \leq n-1.\]
  If $a_{t+1} <1+\km$, then $1-\kp\leq a_i \leq \km$ for $t+1 \leq i
  \leq n$. Since by assumption $n-t\leq t$, it follows
  that
  \[(\underbrace{1,1,\ldots, 1}_{t}, \underbrace{0,0,
    \ldots,0}_{n-t}) \in \parenv*{ \va + \cB } \cap \parenv*{
    \Zero+\cB },\]
  which contradicts the assumption that $\cB$ tiles $\Z^n$ by
  $T$. Hence, $a_{t+1}=1+\km$.

  Now, let $i_0$ be the largest index such that $a_{i_0}=1+\km$. Then
  $i_0-t\geq 1$ as $a_{t+1}=1+\km$. Consider the vector 
  \[\vv\triangleq (\underbrace{1,1,\ldots, 1}_{n-i_0}, \underbrace{0,0,\ldots, 0}_{i_0-t}, a_{n-t+1}, a_{n-t+2}, \ldots, a_{i_0}, \underbrace{0,0,\ldots, 0}_{n-i_0}).\]

  We first compare $\vv$ with $\va$.  Note that $(n-i_0)+(i_0-t)= n-t$
  and $a_i=1$ for $1 \leq i \leq n-t$. Hence, $\vv$ can be obtained
  from $\va$ by changing $n-t$ $a_i$'s to $0$, i.e., those $a_i$'s
  with $n-i_0+1 \leq i \leq n-t$ or $i_0+1 \leq i \leq n$.  Since
  $n-t\leq t$, $a_i=1\leq \km$ for $n-i_0+1 \leq i \leq n-t$, and
  $1-\kp\leq a_i \leq \km$ for $i_0+1 \leq i\leq n$, we have $\vv \in
  \va+\cB$.
  
  Second, we compare $\vv$ with $\Zero$. Note that
  $(i_0-t)+(n-i_0)=n-t$. These two vectors differ in at most $t$
  positions. Hence, $\vv$ can be obtained from $\Zero$ by changing the
  first $n-i_0$ $0$'s to $1$ and the $i$-th $0$ to $a_i$ for $n-t+1
  \leq i\leq i_0$. Since $-\km\leq a_i \leq 1+\km$ for $n-t+1 \leq
  i\leq i_0$, $\kp\geq 1$ and $1+\km \leq \kp$, we have that $\vv \in
  \Zero+\cB$.

  It follows that
  \[ \vv\in (\va+\cB)\cap(\Zero+\cB),\]
  which again contradicts the assumption that $\cB$ tiles $\Z^n$ by
  $T$.
  \qed\end{proof}

  For the last result concerning general tiling, we study the case of
  equal arm length, $\kp=\km$. The method used is an elaboration of
  the one used in the proof of Theorem~\ref{thm:nge2t-1}: instead of
  considering only the all-zero and all-one vectors, we consider a
  third vector as well.

\begin{theorem}
  Let $\kp=\km\geq 2$ and $n> t\geq (4n-2)/5$.  Then for any $n\geq 3$, $\Z^n$ cannot be
  tiled by $\BALL$.
\end{theorem}

\begin{proof} Let $k\triangleq \kp=\km$ and    $\tau\eqdef n-t$.
  Suppose to the contrary that there is a set $T\subseteq \Z^n$ such
  $\cB\eqdef\BALL$ tiles $\Z^n$ by $T$. W.l.o.g., we assume that
  $\Zero \in T$. Since $t\geq (4n-2)/5$ and $n\geq 3$, we have $t\geq n/2$. According to the first three paragraphs in the proof of
  Theorem~\ref{thm:nge2t-1}, we may assume that $\One \in \va+\cB$,
  where
  \[\va\triangleq (\underbrace{1,1,\ldots,1}_{\tau}, \underbrace{1+k, 1+k, \ldots, 1+k}_{i_0-\tau}, a_{i_0+1},\ldots,a_n)    \in T,\]
  with $i_0\geq t+1$, and $1-k\leq a_i \leq k$ for $i_0+1\leq i \leq
  n$. 

  We consider the vector
  \[\vv\triangleq  (\underbrace{2,2,\ldots,2}_{\tau}, \underbrace{1, 1, \ldots, 1}_{i_0-\tau}, a_{i_0+1},\ldots,a_n).\]
  It is not contained in $(\Zero+\cB)\cup (\va+\cB)$ as the Hamming
  distance between $\vv$ and $\Zero$ or $\vv$ and $\va$ is at least
  $i_0 \geq t+1$. We assume that $\vv$ is contained in another ball
  centred at $\vb=(b_1,b_2, \ldots, b_n) \in T$, where $1-k\leq b_i
  \leq 1+k$ for $\tau+1 \leq i \leq i_0$. Let $c\triangleq \abs{\set{ i ;
    \tau+1\leq i\leq i_0, b_i=1+k}}$.  We proceed in the following two
  cases.
  
  \begin{enumerate}
  \item
    If $c \leq i_0-3\tau$, by interchanging all the coordinates
    between $\tau+1$ and $i_0$, we may assume that $1-k\leq b_i \leq
    k$ for $i_0-2\tau+1\leq i\leq i_0$. We consider the vector
    \[\vx\triangleq (\underbrace{2,\ldots,2}_{\tau}, \underbrace{1,1,\ldots,1}_{i_0-3\tau}, \underbrace{0,\ldots,0}_{\tau}, \underbrace{b_{i_0-\tau+1} \ldots, b_{i_0}}_{\tau}, a_{i_0+1}, \ldots, a_n).\]
    We first compare $\vx$ with $\Zero$. These two vectors agree in at
    least $\tau=n-t$ positions. Noting that $k\geq 2$, $1-k\leq b_i
    \leq k$ for $i_0-\tau+1\leq i\leq i_0$ and $1-k\leq a_i \leq k
    \textrm{ for } i_0+1\leq i \leq n$, we have $\vx \in
    \Zero+\cB$. Second, we compare $\vx$ with $\vb$. They differ in
    the first $i_0-\tau$ positions and the last $n-i_0$ positions, and
    so in total $n-\tau=t$ positions. Noting that $\vx$ and $\vv$
    agree in the first $i_0-2\tau$ positions and the last $n-i_0$
    positions and $\vv\in \vb+\cB$, the symbols of $\vx$ in these
    positions can be obtained from the corresponding symbols of $\vb$
    by adding or subtracting up to $k$ units.  For the remaining
    $\tau$ positions where $i_0-2\tau+1\leq i\leq i_0-\tau$, we have
    $1-k\leq b_i \leq k$.  It follows that $\vx \in \vb +\cB$ and then
    \[ \vx\in (\vb+\cB)\cap(\Zero+\cB).\]

  \item
    If $c > i_0-3\tau$, we may assume that $b_i=1+k$ for $\tau+1 \leq
    i \leq i_0-2\tau+1$.  Consider the vector
    \[\vy\triangleq (\underbrace{2,\ldots,2}_{\tau}, \underbrace{1+k,\ldots,1+k}_{i_0-3\tau+1}, 1,\ldots,1, a_{i_0+1}, \ldots,a_n).\]
    We first compare $\vy$ with $\vb$. These two vectors differ in the
    first $\tau$ positions and the last $n-i_0+2\tau-1$
    positions. Since $t\geq (4n-2)/5$, they differ in a total of
    $n-i_0+3\tau-1 \leq 4\tau-2 = 4 n-2 -4t \leq t$ positions. Noting
    that $\vy$ and $\vv$ agree in these positions and $\vv \in
    \vb+\cB$, we have $\vy \in \vb+\cB$. Second, we compare $\vy$ with
    $\va$. They differ in a total of $\tau+
    (i_0-\tau)-(i_0-3\tau+1)=3\tau-1=  3n-3t - 1$ positions.  Note that
    $t\geq (4n-2)/5 \geq (3n-1)/4$ as $n\geq 3$. Thus we have
    $3n-3t - 1\leq t$. Furthermore, in these $3n-3t - 1$ positions, the
    corresponding symbols differ by at most $k$ units. It follows that $\vy
    \in \va +\cB$ and then
    \[ \vy\in (\va+\cB)\cap( \vb+\cB).\]
  \end{enumerate}
  
  In both cases above we obtain a contradiction to the assumption that
  $\cB$ tiles $\Z^n$ by $T$.
\qed\end{proof}

\subsection{Nonexistence of Lattice Tilings}

We now turn to the more specific case of lattice tilings. Some of the
nonexistence results presented in this section are stated as necessary
conditions. The main tool used is Theorem~\ref{th:tiletolat}, and the
algebraic study of the $t$-splitting. 
We begin with the lattice-tiling
equivalent of Theorem~\ref{th:non1}.

\begin{theorem}
For any $n\geq t+1$, and $\kp\geq \km\geq 0$ not both $0$, if $\BALL$ lattice-tiles $\Z^n$ then
\[\sum_{i=1}^t \binom{n}{i}(\kp+\km)^{i-1} \geq (\km+1)^t.\]
\end{theorem}

\begin{proof}  For $t=1$, see \cite[Theorem 11]{Sch12}. In the following, we focus on  the cases  $t\geq 2$.
  Assume that $\cB\eqdef\cB(n,t,\kp,\km)$ lattice-tiles $\Z^n$. By
  Theorem~\ref{th:tiletolat} there is an Abelian group $G$ with
  $\abs{G}= \sum_{i=0}^t \binom{n}{i} (\kp+\km)^i$ and a subset
  $S=\set{s_1,s_2,\ldots, s_n} \subseteq G$ such that $G=M\splt_t S$,
  where $M\eqdef[-\km,\kp]^*$.


  We first claim that for all $2\leq i_1<i_2<\cdots <i_t \leq n$ there
  are integers $x_1^{i_1,i_2,\ldots, i_t}, x_{i_1}^{i_1,i_2,\ldots,
    i_t}, \ldots, x_{i_t}^{i_1,i_2,\ldots, i_t}$ such that $0 
  \leq x_1^{i_1,i_2,\ldots, i_t} \leq \floor*{
    \frac{\abs{G}}{(\km+1)^t}}$, $\abs{x_{i_j}^{i_1,i_2,\ldots,i_t}}\leq \km$
  for $j=1,2,\ldots,t$, and
  \[s_1x_1^{i_1,i_2,\ldots, i_t}+s_{i_1} x_{i_1}^{i_1,i_2,\ldots, i_t} +\cdots + s_{i_t} x_{i_t}^{i_1,i_2,\ldots, i_t} =0.\]

  To prove this, fix $i_1,i_2,\ldots, i_t$ and look at the integers $0
  \leq a_1 \leq \floor*{ \frac{\abs{G}}{(\km+1)^t}}$, $0\leq a_{i_j}\leq
  \km$ for $j=1,2,\ldots,t$ and the sums $s_1a_1+s_{i_1}a_{i_1}+\cdots
  +s_{i_t}a_{i_t}$. Since
  \begin{align*}
    \parenv*{  \floor*{  \frac{\abs{G}}{(\km+1)^t}} +1 }(\km+1)^t &\geq \abs{G} - ((\km+1)^t-1) +(\km+1)^t \\
    &=\abs{G}+1 >\abs{G},
  \end{align*}
  by the pigeonhole principle there exist two sequences of integers,
  $(b_1,b_{i_1}, \ldots, b_{i_t})$ and $(c_1,c_{i_1}, \ldots,
  c_{i_t})$, such that
  \[s_1b_1+s_{i_1}b_{i_1}+\cdots +s_{i_t}b_{i_t} =
  s_1c_1+s_{i_1}c_{i_1}+\cdots +s_{i_t}c_{i_t}.\]

  Assume, w.l.o.g., that $b_1\geq c_1$ and define $d_1\eqdef b_1-c_1$ and
  $d_{i_j}\eqdef b_{i_j}-c_{i_j}$ for $j=1,2,\ldots,t.$ We now get
  \[s_1d_1+s_{i_1}d_{i_1}+\cdots +s_{i_t}d_{i_t} = 0,\]
  where $(d_1,d_{i_1},\ldots, d_{i_t})\neq (0,0,\ldots,0)$. In addition 
  \[0\leq d_1 \leq  \floor*{  \frac{\abs{G}}{(\km+1)^t}} \text{\ and }  \abs{d_{i_j}}\leq \km \text{\ for }  j=1,2,\ldots,t,\]
which prove  our claim.
 
 Now, if $ \floor*{  \frac{\abs{G}}{(\km+1)^t}} \leq \kp+\km$, then $0\leq d_1 \leq  \floor*{  \frac{\abs{G}}{(\km+1)^t}} \leq \kp+\km$, and
  \[s_1\kp+s_{i_1}d_{i_1}+\cdots+s_{i_{t-1}}d_{i_{t-1}} = s_1 (\kp-d_1)
  -s_{i_t}d_{i_t},\]
  which contradicts the fact that $G=M\splt_t S$, since $t\geq 2$. 
  
  Hence, we have that 
  \begin{equation*}
    \kp+\km+1\leq \floor*{  \frac{\abs{G}}{(\km+1)^t}}.
  \end{equation*}
  It follows  that
  \begin{align*}
    (\km+1)^t &\leq \frac{\abs{G}}{\kp+\km+1 } <  \frac{\sum_{i=0}^t \binom{n}{i} (\kp+\km)^i }{ \kp+\km} \\
    &= \sum_{i=1}^t \binom{n}{i} (\kp+\km)^{i-1}+\frac{1}{\kp+\km}.
  \end{align*}
  Since both $(\km+1)^t$ and $\sum_{i=1}^t \binom{n}{i}
  (\kp+\km)^{i-1}$ are integers and $\frac{1}{\kp+\km}$ is at most
  $1$, we have
  \[(\km+1)^t \leq \sum_{i=1}^t \binom{n}{i}(\kp+\km)^{i-1}.\]
\qed\end{proof}

Using similar arguments to the previous theorem, the next one
specializes in the case of $n\geq 2t$.

\begin{theorem}\label{thm:lattice-sym}
Let $n\geq 2t$, and $\kp\geq\km\geq 0$.  If $\BALL$ lattice-tiles $\Z^n$ then
\[\frac{(\km+1)^2}{\kp+\km+1} < \binom{n}{t}^{1/t},\]
\end{theorem}

\begin{proof}
  If $\BALL$ lattice-tiles $\Z^n$, by Theorem~\ref{th:tiletolat} there
  is an Abelian group $G$ with $\abs{G}= \sum_{i=0}^t \binom{n}{i}
  (\kp+\km)^i$ and a subset $S=\set{s_1,s_2,\ldots, s_n} \subseteq G$
  such that $G=M\splt_t S$, where $M\eqdef[-\km,\kp]^*$. We consider
  the sums
  \[x_1s_1+x_2s_2+\cdots + x_t s_t+y_1s_{t+1}+y_2s_{t+2} +\cdots + y_t s_{2t},\]
  where $0\leq x_i < \frac{\kp+\km+1}{\km+1} \binom{n}{t}^{1/t}$ and
  $0\leq y_i \leq \km$ for $i=1,2,\ldots, t$. The total number of such
  sums is at least $\binom{n}{t}(\kp+\km+1)^t$.  Noting that
  \[\binom{n}{t}(\kp+\km+1)^t =  \sum_{i=0}^t \binom{n}{t}\binom{t}{i} (\kp+\km)^i  > \sum_{i=0}^t \binom{n}{i} (\kp+\km)^i=\abs{G},\]
  there are two sums which are equal. Namely, there are
  \[\va, \va'\in \sparenv*{0, \ceil*{\frac{\kp+\km+1}{\km+1} \binom{n}{t}^{1/t}}-1}^t
  \quad\text{and}\quad \vb,\vb'\in [0,\km]^t,\]
  with $(\va,\vb)\neq (\va',\vb')$, such that
  \begin{multline*}
    \va \cdot (s_1,s_2,\ldots, s_t)+\vb \cdot(s_{t+1},s_{t+2}, \ldots, s_{2t})\\
    = \va' \cdot (s_1,s_2,\ldots, s_t)+\vb' \cdot(s_{t+1},s_{t+2}, \ldots, s_{2t}).
  \end{multline*}
  Let $\vc=\va-\va'$ and $\vd=\vb'-\vb$. Rearranging the terms, we have 
  \[\vc \cdot (s_1,s_2,\ldots, s_t)=\vd \cdot(s_{t+1},s_{t+2}, \ldots, s_{2t}).\] 
  Since $\vc \in \sparenv*{- \ceil*{\frac{\kp+\km+1}{\km+1}
      \binom{n}{t}^{1/t}}+1, \ceil*{\frac{\kp+\km+1}{\km+1}
      \binom{n}{t}^{1/t}}-1}^t$, $\vd\in [-\km,\km]^t$, and
  $(\vc,\vd)\neq(\Zero,\Zero)$, to avoid contradicting the assumption
  $G=M\splt_t S$, necessarily
  \[\km< \ceil*{\frac{\kp+\km+1}{\km+1} \binom{n}{t}^{1/t}}-1,\]
which implies
  \[\km \leq \frac{\kp+\km+1}{\km+1} \binom{n}{t}^{1/t}-1.\]
  The claim now follows by rearranging.
  \qed\end{proof}

  Theorem~\ref{thm:lattice-sym} is particularly useful in an
  asymptotic regime where $t=\Theta(n)$, as shown in the following
  corollary.

  \begin{corollary}
    If $\alpha \leq \frac{t}{n} \leq \frac{1}{2}$, $\kp\geq\km\geq 0$, and
    \[ \frac{(\km+1)^2}{\kp+\km+1} \geq \frac{e}{\alpha},\]
    then $\BALL$ does not lattice-tile $\Z^n$.
  \end{corollary}

  \begin{proof}
    We observe that
    \[
    \frac{(\km+1)^2}{\kp+\km+1} \geq \frac{e}{\alpha}
    \geq \frac{ne}{t}>\binom{n}{t}^{1/t},\]
    and the claim now follows by Theorem~\ref{thm:lattice-sym}.
\qed\end{proof}

We continue on to a few more specific cases. The next two theorems
deal with the analogue of semi-crosses when $t=1$, namely, the case of
$\km=0$. First a technical lemma is required.

\begin{lemma}\label{lem-settwice}
Let $A\subseteq [0,\binom{n}{t}-1]$ be a subset of size $\binom{n-1}{t}$.  If 
\[\parenv*{\frac{n}{4t}-1}\binom{n-1}{t-1}> \frac{1}{2},\]
 then $A$ contains two elements $a$ and $b$ such that $b=2a\neq 0$.
\end{lemma}

\begin{proof}
Define \[m\eqdef  \floor*{\frac{1}{2}\parenv*{\binom{n}{t} -1}},\] and
\[B \eqdef \bigcup_{i=1}^{m}  \{i,2i\}.\] 
Then $B$ is a subset of $[0,\binom{n}{t}-1]$ with
$\abs{B}=2m-\floor{m/2}$. Consider the intersection of $A$ and $B$,
\begin{align*}
  \abs{A\cap B} &= \abs{A}+\abs{B} -\abs{A\cup B} \geq \binom{n-1}{t} + 2m-\lfloor m/2\rfloor - \binom{n}{t} \\
  &\geq m+ m/2 - \binom{n-1}{t-1}
  \geq m+ \frac{1}{2} \cdot\frac{ \binom{n}{t} -2}{2}- \binom{n-1}{t-1} \\
  &= m +  \parenv*{\frac{n}{4t}-1}\binom{n-1}{t-1} -\frac{1}{2} > m.
\end{align*}
Then $A$ contains at least one pair, $i$ and $2i$, from $B$. 
\qed\end{proof}

\begin{theorem}
Let $2\leq t < n/4$ and $\kp> \km=0$.  Then $\cB(n,t,\kp,0)$
cannot lattice-tile $\Z^n$ when
\[\kp\geq  2 \binom{n}{t}-2.\] 
\end{theorem}

\begin{proof}
  By Theorem~\ref{th:tiletolat}, suppose to the contrary that there is
  an Abelian group $G$ with $\abs{G}= \sum_{i=0}^t \binom{n}{i} \kp^i$ and
  a subset $S=\set{s_1,s_2,\ldots, s_n} \subseteq G$ such that
  $G=M\splt_t S$, where $M\eqdef[1,\kp]$. We consider the sums
  \[x_1s_1+x_2s_2+\cdots + x_t s_t+x_{t+1}s_{t+1},\]
  where $0\leq x_1 < \binom{n}{t}$ and $0\leq x_i \leq \kp$ for
  $i=2,3,\ldots, t+1$. The total number of such sums is $\binom{n}{t}
  (\kp+1)^t$.  Noting that
  \[ \binom{n}{t} (\kp+1)^t =  \sum_{i=0}^t  \binom{n}{t}  \binom{t}{i}  \kp^i    > \sum_{i=0}^t \binom{n}{i} \kp^i=\abs{G},\]
  there are two sums which are equal. Namely, there are two distinct
  vectors, $\va=(a_1,a_2,\ldots,a_{t+1})$ and $\va'=(a_1',a_2',\ldots,
  a_{t+1}')$, from $[0,\binom{n}{t}-1]\times [0, \kp]^{t}$, such that
  \[ a_1 s_1+a_2s_2+\ldots+a_{t+1} s_{t+1}=a_1' s_1+a_2's_2+\ldots+a_{t+1}'s_{t+1}.\]

  W.l.o.g., assume $a_1 \geq a_1'$.  Let $b_i=a_i-a_i'$ for
  $i=1,2,\ldots, t+1$. Rearranging the terms, we have
  \begin{equation}\label{eq-idc-1}
    b_1 s_1+b_2s_2+\ldots+b_{t+1} s_{t+1}=0,
  \end{equation}
  where $\vb=(b_1,b_2,\ldots, b_{t+1})$ is a non-zero vector from
  $[0,\binom{n}{t}-1]\times [-\kp,\kp]^t$. Since $\binom{n}{t}-1 \leq
  \kp$, to avoid contradicting the assumption $G=M\splt_t S$,
  necessarily $b_i\geq 0$ for all $i=2,3,\ldots,t+1$, i.e., $\vb \in
  [0, \kp]^{t+1}$.

  We now claim that there is a non-zero vector $\vv \in [0,
    \binom{n}{t}-1]^{t+1}$ such that $\vv\cdot (s_1,s_2,\ldots,
  s_{t+1})=0$. As a first step, we show that there is a non-zero vector
  $\vv=(v_1,v_2,\ldots, v_{t+1})\in [0, \kp]^{t+1}$ such that $v_1,v_2
  < \binom{n}{t}$ and $\vv\cdot (s_1,s_2,\ldots,
  s_{t+1})=0$. In~\eqref{eq-idc-1}, if $b_2 < \binom{n}{t}$, then
  $\vb$ is the desired vector. Otherwise, $b_2 \geq \binom{n}{t}$. By
  symmetry (repeating the same arguments arriving in
  \eqref{eq-idc-1}), there is a non-zero vector $\vc=(c_1,c_2,\ldots,
  c_{t+1}) \in [0, \kp]^{t+1}$ with $c_2 < \binom{n}{t}$ such that
  $\vc\cdot (s_1,s_2,\ldots, s_{t+1})=0$. If $c_1 \geq \binom{n}{t}$,
  we consider the equation
  \[  (b_2-c_2)s_2+ b_3s_3+b_4s_4+ \cdots + b_{t+1}s_{t+1}= (c_1-b_1)s_1+ c_3s_3+c_4s_4+ \cdots + c_{t+1}s_{t+1},\]
  which is obtained by rearranging $\vb \cdot (s_1,s_2,\ldots,
  s_{t+1}) = \vc \cdot (s_1,s_2,\ldots, s_{t+1}).$ Note that
  $0<c_1-b_1 \leq \kp$, $0<b_2-c_2\leq \kp$, and $b_i,c_i\in [0,\kp]$
  for all $i=3,4,\ldots,t+1$. This contradicts the assumption
  $G=M\splt_t S$. Thus, necessarily, $c_1 <\binom{n}{t}$ and $\vc$ is
  the desired vector.  By using induction on the first $j$ elements,
  $s_1,s_2,\ldots, s_j$, we are able to show our claim.

  Extending the arguments presented thus far, for any $2\leq
  i_1<i_2<\cdots <i_t \leq n$, there is a non-zero vector
  \[\vv^{i_1,i_2,\ldots, i_t} =(v_1^{i_1,i_2,\ldots, i_t}, v_{i_1}^{i_1,i_2,\ldots,i_t}, \ldots, v_{i_t}^{i_1,i_2,\ldots, i_t}) \in \left[0, \binom{n}{t}-1\right]^{t+1}\] such that  
   \begin{equation*}\label{eq-ls}
   v_1^{i_1,i_2,\ldots, i_t}s_1+ v_{i_1}^{i_1,i_2,\ldots, i_t} s_{i_1}+\cdots +  v_{i_t}^{i_1,i_2,\ldots, i_t}s_{i_t} =0.
   \end{equation*}

   Take any $2\leq i_1'<i_2'<\cdots <i_t' \leq n$ such that
   $(i_1,i_2,\dots,i_t)\neq(i_1',i_2',\dots,i_t')$.  If there are two
   integers $v_1^{i_1,i_2,\ldots, i_t}$ and $v_1^{i_1',i_2',\ldots,
     i_t'}$ which are equal, then we have
   \begin{multline*}
     v_{i_1}^{i_1,i_2,\ldots, i_t} s_{i_1}+v_{i_2}^{i_1,i_2,\ldots, i_t} s_{i_2}+\cdots +  v_{i_t}^{i_1,i_2,\ldots, i_t}s_{i_t}\\
     =v_{i_1'}^{i_1',i_2',\ldots, i_t'} s_{i_1'}+v_{i_2'}^{i_1',i_2',\ldots, i_t'} s_{i_2'}+\cdots +  v_{i_t'}^{i_1',i_2',\ldots, i_t'}s_{i_t'}.
   \end{multline*}
   To avoid contradicting the assumption that $G=M\splt_t S$,
   necessarily, $v_{i_j}^{i_1,i_2,\ldots,
     i_t}=v_{i_j'}^{i_1',i_2',\ldots, i_t'}=0$ for all $1\leq j \leq
   t$, which in turn implies $v_{1}^{i_1,i_2,\ldots, i_t}=0$. This
   contradicts the fact that $\vv^{i_1,i_2,\ldots, i_t}$ is a non-zero
   vector.  Therefore, the $\binom{n-1}{t}$ integers
   $v_1^{i_1,i_2,\ldots, i_t}$ must be pairwise distinct.
    
   Note that when $2\leq t < n/4$, we have  
   \[\parenv*{\frac{n}{4t}-1}\binom{n-1}{t-1}> \frac{1}{2}.\] 
   By Lemma~\ref{lem-settwice}, there are
   $v_1^{i_1,i_2,\ldots, i_t}$ and $v_1^{i_1',i_2',\ldots, i_t'}$ such
   that $v_1^{i_1,i_2,\ldots, i_t} =2 v_1^{i_1',i_2',\ldots,
     i_t'}\neq 0$. Therefore,
   \begin{multline*}
     v_{i_1}^{i_1,i_2,\ldots, i_t}
   s_{i_1}+v_{i_2}^{i_1,i_2,\ldots, i_t} s_{i_2}+\cdots +
   v_{i_t}^{i_1,i_2,\ldots, i_t}s_{i_t}\\
   =2v_{i_1'}^{i_1',i_2',\ldots,
     i_t'} s_{i_1'}+2v_{i_2'}^{i_1',i_2',\ldots, i_t'} s_{i_2'}+\cdots
   + 2 v_{i_t'}^{i_1',i_2',\ldots, i_t'}s_{i_t'}.
   \end{multline*}
   Note that $\{i_1,i_2,\ldots, i_t\}\neq \{i_1',i_2',\ldots, i_t'\}$,
   $0\leq v_{i_j}^{i_1,i_2,\ldots, i_t} \leq \binom{n}{t}-1 \leq \kp$
   and $0\leq 2v_{i_j'}^{i_1',i_2',\ldots, i_t'} \leq 2\binom{n}{t}-2
   \leq \kp$. To avoid contradicting the assumption, necessarily
   $v_{i_j}^{i_1,i_2,\ldots, i_t}=v_{i_j'}^{i_1',i_2',\ldots, i_t'}=0$
   for all $1\leq j\leq t$, and so $v_{1}^{i_1,i_2,\ldots,
     i_t}=0$. This contradicts the fact that $\vv^{i_1,i_2,\ldots,
     i_t}$ is a non-zero vector, which completes our proof.
   \qed\end{proof}

   Unlike the other proofs in this section, the next one uses a
   geometric argument.
   
\begin{theorem}
Let $\frac{2}{3}(n-1)\leq t \leq n-3$. Then $\cB(n,t,\kp,0)$ cannot
lattice-tile $\Z^n$ when $\kp\geq 2$.
\end{theorem}

\begin{proof}
  Suppose to the contrary that there is a lattice $\Lambda\subseteq
  \Z^n$ such that $\cB$ tiles $\Z^n$ by $\Lambda$.    According to the first
  two paragraphs in the proof of Theorem~\ref{thm:nge2t-1}, we may
  assume that $\One \in \va+\cB$, where
  \[\va\eqdef (\underbrace{1,1,\ldots,1}_{t+1}, a_{t+2},\ldots,a_n)   \in \Lambda,\]
  where $1-\kp\leq a_i \leq 1$ for $t+2 \leq i \leq n$. 

  Let $\tau \eqdef n-t$. The assumption $\frac{2}{3}(n-1)\leq t \leq
  n-3$ implies $\tau \geq 3$ and $2\tau-2\leq t$. We consider the
  vector
  \[\vv \eqdef (\underbrace{0,0,\ldots,0}_{\tau-1}, \underbrace{1,1,\ldots,1}_{t+1}).\]
  Since $\wt(\vv)= t+1$ and $\tau-1\geq 1$, neither $\cB$ nor
  $\va+\cB$ contains $\vv$. Thus there is another vector
  \[\vb=(b_1,b_2,\ldots, b_n)\in \Lambda\]
  such that $\vv \in \vb+\cB$, where $-\kp\leq b_i \leq 0$ for $1\leq
  i\leq \tau-1$ and $1-\kp \leq b_i \leq 1$ for $\tau \leq i \leq
  n$. In the following, we further narrow down the range of $b_i$.
  \begin{enumerate}
  \item
    $b_i=1$ for all $\tau \leq i \leq n$. Otherwise, w.l.o.g., assume
    $b_\tau\leq 0$. Note that $\vv\in \vb+\cB$. Then
    $(\underbrace{0,0,\ldots,0}_{\tau},\underbrace{1,1,\ldots,1}_t)\in
    \vb+\cB$, contradicting
    $(\underbrace{0,0,\ldots,0}_{\tau},\underbrace{1,1,\ldots,1}_t)
    \in \cB$.
  \item
    There is at least one $b_i=-\kp$ for some $1\leq i \leq
    \tau-1$. Otherwise, $-\kp< b_i \leq 0$ for all $1\leq i\leq
    \tau-1$. Note that $\tau-1 \leq t$ and we have shown $b_i=1$ for
    all $\tau \leq i \leq n$. It follows that $\One \in \vb+\cB$,
    which contradicts $\One\in \va+\cB$.
  \end{enumerate}
  
  According to the argument above, by permuting the first $\tau-1$
  elements of $\vb$, we may assume
  \[ \vb=(-\kp,\underbrace{0,\ldots,0}_p,  \underbrace{b_{p+2,},\ldots, b_{\tau-1}}_{q}, \underbrace{1,1,\ldots,1}_{t+1}),\]
  where $p,q\geq 0$, $p+q=\tau-2$ and $-\kp\leq b_i \leq -1$ for $p+2
  \leq i\leq \tau-1$.

  Now, for $0\leq \ell \leq p$, define
  \[\vu_\ell \eqdef(1,\underbrace{0,0,\ldots,0}_p,\underbrace{1,1,\ldots,1}_{q+ \ell}, \underbrace{0,0,\ldots,0}_{q+1}, \underbrace{1,1,\ldots,1}_{n-p-2q-\ell-2}). \]
  There are $n-p-q-1=t+1$ ones in $\vu_\ell$ and so $\vu_\ell$ is not
  contained in $\cB$. Noting that $1+p+q+\ell+q+1 \leq
  2+2(p+q)=2\tau-2\leq t+1$, there are $\tau-1$ zeros in the first
  $t+1$ entries of $\vu_\ell$, and so $\vu_\ell\not\in \va+\cB$.  The
  first entry of $\vu_\ell$ is 1 while the first entry of $\vb$ is
  $-\kp$. Thus, $\vu_\ell\not\in \vb+\cB$.

  Assume $\vu_\ell\in \vc_\ell+\cB$ for some $\vc_\ell\in
  \Lambda$. According to the argument above, necessarily $\vc_\ell
  \not \in \{\Zero, \va,\vb\}$. Since both $\vu_\ell$ and $\vv$ have
  $\tau-1$ zeros in the first $t+1$ entries and ones in all the other
  entries and $\va$ has ones in the first $t+1$ entries, according to
  the symmetry, $\vc_\ell$ has the same form as $\vb$, namely,
  \[\vc_\ell =(1,\underbrace{*,*,\ldots,*}_p,\underbrace{1,1,\ldots,1}_{q+ \ell}, \underbrace{*,*,\ldots,*}_{q+1}, \underbrace{1,1,\ldots,1}_{n-p-2q-\ell-2}),\]
  where the entries marked with $*$ are in $[-\kp,0]$ and at least one
  of them is $-\kp$.

  We claim that all the last $q+1$ entries marked with $*$ in
  $\vc_\ell$ should be $0$. Otherwise, w.l.o.g., assume the first of
  them is negative, i.e.,
  \[\vc_\ell =(1,\underbrace{*,*,\ldots,*}_p,\underbrace{1,1,\ldots,1}_{q+ \ell}, \underbrace{-x,*,\ldots,*}_{q+1}, \underbrace{1,1,\ldots,1}_{n-p-2q-\ell-2}),\]
  where $1 \leq x\leq \kp$.  Then
  \begin{multline*}
    \vb+\vc_\ell=(1-\kp,\underbrace{*,*,\ldots,*}_p,\underbrace{b_{p+2}+1, b_{p+3}+1,\ldots,b_{\tau-1}+1}_{q}, \underbrace{2,2,\ldots,2}_{\ell},
    \\
    \underbrace{1-x,\oast,\ldots,\oast}_{q+1}, \underbrace{2,2,\ldots,2}_{n-p-2q-\ell-2}),
  \end{multline*}
  where the entries marked with $*$ are in $[-\kp,0]$ and the entries marked with $\oast$ are in
  $[1-\kp,1]$.  Note that $-\kp\leq b_i \leq -1$ for $p+2 \leq i\leq
  \tau-1$, and $1+p+q+q+1 \leq 2(\tau-2)+2=2\tau-2\leq t$. It follows
  that
  \[ (\underbrace{0,0,\ldots,0}_{1+p+q = \tau-1}, \underbrace{2,2,\ldots,2}_\ell, 0, \underbrace{1,1,\ldots,1}_q, \underbrace{2,2,\ldots,2}_{n-p-2q-\ell-2} ) \in  \vb+\vc_\ell+\cB.\]
  Since $\kp\geq 2$, the vector above is also contained in $\cB$. Then
  we got $\vb+\vc_\ell=\Zero,$ which contradicts that the first entry
  of $\vb+\vc_\ell$ is $1-\kp\leq -1.$ Therefore,
  \[\vc_\ell =(1,\underbrace{*,*,\ldots,*}_p,\underbrace{1,1,\ldots,1}_{q+ \ell}, \underbrace{0,0,\ldots,0}_{q+1}, \underbrace{1,1,\ldots,1}_{n-p-2q-\ell-2}).\]
  Recall that the entries marked with $*$ are in $[-\kp,0]$ and at
  least one of them is $-\kp$. Necessarily $p\geq 1$. Since there are
  $p+1$ choices of $\ell$, at least two vectors, say $\vc_{\ell_1}$
  and $\vc_{\ell_2}$, have $-\kp$ in the same entry. By permuting the
  $p$ entries marked with $*$, assume both $\vc_{\ell_1}$ and
  $\vc_{\ell_2}$ have $-\kp$ in the first entry marked with $*$.  Then
  \[ (1,\underbrace{-\kp,0,\ldots,0}_p,\underbrace{1,1,\ldots,1}_{n-p-1}) \in (\vc_{\ell_1}+\cB)\cap (\vc_{\ell_2}+\cB), \]
  as $p-1+q+1=\tau-2 \leq t$.  It follows that
  $\vc_{\ell_1}=\vc_{\ell_2}$.  W.l.o.g., assume $\ell_1 <
  \ell_2$. Then the $(n-p-2q-\ell_2-1)$-th entry, from the right side,
  of $\vc_{\ell_2}$ is $0$, while the corresponding entry of
  $\vc_{\ell_1}$ is $1$, a contradiction. \qed\end{proof}

  Continuing our specialization, we turn to tackle the case of $t=2$,
  and present a strong restriction on the dimension $n$.
  
\begin{theorem}
  For any $\kp\geq\km\geq 0$, if $\cB(n,2,\kp,\km)$ lattice-tiles
  $\Z^n$ and also $\abs*{\cB(n,2,\kp,\km)}$ is even, then
  \[ n=\frac{4\ell^2-(\kp+\km-3)^2+8}{4(\kp+\km)},\]
  for some $\ell\in\Z$.
\end{theorem}
\begin{proof}
  By Theorem~\ref{th:tiletolat} there exists an Abelian group $G$
  whose size is $\abs*{G}=\abs*{\cB(n,2,\kp,\km)}$ such that
  $G=M\splt_2 S$ for some $S\subseteq G$, $\abs*{S}=n$, where $M\eqdef
  [-\km,\kp]^*$. Since $G$ is Abelian and of even order, necessarily
  $G=\Z_{2^r}\times G'$, for some $r\geq 1$. We may therefore write
  any element $g\in G$ as a pair $(a,b)$ where $a\in\Z_{2^\ell}$ and
  $b\in G'$, and we say $g$ is \emph{even} if $a\equiv 0\pmod{2}$, and
  \emph{odd} otherwise.

  Denote by $n_1$ the number of odd elements in $S$. Additionally,
  denote by $m_0\eqdef \floor{\kp/2}+\floor{\km/2}$ (respectively,
  $m_1\eqdef \ceil{\kp/2}+\ceil{\km/2}$) the number of even
  (respectively, odd) numbers in $M$.

  Let us examine how the
  $\frac{1}{2}\parenv*{\binom{n}{2}(\kp+\km)^2+n(\kp+\km)+1}$ odd
  elements of $G$ are obtained via the $2$-splitting. There are three
  possible ways:
  \begin{enumerate}
  \item
    An odd element in $S$ times an odd number in $M$.
  \item
    An odd element in $S$ times an odd number in $M$,
    plus an even element in $S$ times any number from $M$.
  \item
    An odd element in $S$ times an odd number in $M$, plus a different
    odd element in $S$ times an even number from $M$.
  \end{enumerate}
  Thus,
  \begin{multline*}
    n_1 m_1+n_1 m_1 (n-n_1) (m_0+m_1)+n_1 m_1(n_1-1)m_0\\
    =\frac{1}{2}\parenv*{\binom{n}{2}(m_0+m_1)^2+n(m_0+m_1)+1}.
  \end{multline*}
  Solving for $n_1$ we obtain
  \begin{equation}
    \label{eq:sqrt}
    n_1 = \frac{n(m_0+m_1)-m_0+1\pm\sqrt{n(m_1^2-m_0^2)+m_0^2-2m_0-1}}{2m_1}.
  \end{equation}
  We recall that $m_0+m_1=\kp+\km$. Additionally, we note that
  \[\abs*{\cB(n,2,\kp,\km)}=\binom{n}{2}(\kp+\km)^2+n(\kp+\km)+1\]
  is even, which implies that $\kp+\km$ is odd, and then
  $m_1-m_0=1$. It follows that
  $m_1^2-m_0^2=(m_1-m_0)(m_1+m_0)=m_1+m_0=\kp+\km$. Substituting back
  in~\eqref{eq:sqrt}, we use the fact that the square root must be an
  integer $\ell\in\Z$ to obtain the desired claim after some simple
  rearranging.  \qed\end{proof}

  Finally, we focus on the smallest case not studied before -- tiling
  $\cB(n,2,1,0)$. In this case, by a careful study of the possible
  group splittings we obtain a full classification of possible
  tilings. We require some structural lemmas first. These hold for a
  weaker structure than a $t$-splitting: If in
  Definition~\ref{def:split} only the first condition holds, we denote
  it as $G\geq M\splt_t S$.

  \begin{lemma}\label{lm:t=2-difference1}
    Suppose that $G\geq \set{1} \splt_2 S$. Let $n=\abs{S}$. Consider
    the $(n+1)n$ differences $s-s'$, where $s,s' \in S \cup \set{0}$
    and $s\neq s'$. If there are two differences which are equal, then
    they must have the form
    \[s_i-s_j=s_k-s_i,\]
    for some $s_i,s_j$ and $s_k \in S\cup \set{0}$.  Furthermore, if
    $s_i=0$, then we must have $s_j=s_k$.
  \end{lemma}

  \begin{proof}
    Assume that there are two distinct pairs $(s_i,s_j), (s_k,s_\ell)
    \in (S\cup\set{0})^2$ with $s_i\neq s_j$ and $s_k\neq s_\ell$ such
    that
    \[s_i-s_j=s_k-s_\ell.\]
    Rearranging the terms, we have 
    \[s_i+s_\ell=s_k+s_j.\]
    Since $G\geq \set{1} \splt_2 S$, $(s_i,s_j)\neq (s_k,s_\ell)$ and
    $\set{s_i,s_\ell}\neq \set{s_k,s_j}$, either $s_i=s_\ell$ or
    $s_k=s_j$. Then the conclusion follows.
  \qed\end{proof}

  \begin{lemma}\label{lm:t=2-difference2}
    Suppose that $G\geq \set{1} \splt_2 S$. For each $s_i \in S$,
    there is at most one unordered  pair $\set{s_j,s_k} \subset S\cup\set{0}$
    with $s_j\neq s_k$ such that
    \[s_i-s_j=s_k-s_i.\]
  \end{lemma}

  \begin{proof}
    Suppose that there is another pair $\set{s_j',s_k'}$ with
    $s_j'\neq s_k'$ such that $s_i-s_j'=s_k'-s_i$. Then
    \[2s_i=s_j+s_k=s_j'+s_k'.\]
    Since $s_j\neq s_k$, $s_j'\neq s_k'$ and $G\geq \set{1} \splt_2
    S$, necessarily $\set{s_j,s_k}=\set{s_j',s_k'}$.
  \qed\end{proof}

  For an Abelian group $G$, let $m_2(G)$ be the number of elements of
  order $2$ in $G$, i.e.,
  \[m_2(G)\eqdef  \abs*{\set*{x\in G ; x\neq 0, 2x=0}}.\]

  \begin{lemma}\label{lem:order2}
    Suppose that $G\geq \set{1} \splt_2 S$, and let
    $n\eqdef\abs{S}$. Then we have
    \[ \abs{G}+m_2(G)\geq n^2 -n+1.\]
  \end{lemma}

\begin{proof}
  Denote
  \[\Delta \eqdef \set{(s,s') ; s,s'\in S\cup\set{0} \text{ and } s\neq s' }.\]
  According to Lemma~\ref{lm:t=2-difference2}, for each $s_i \in S$,
  there is at most one unordered pair $\set{s_j,s_k} \subset S\cup \set{0}$ with
  $s_j\neq s_k$ such that $s_i-s_j=s_k-s_i$ (and so
  $s_i-s_k=s_j-s_i$). If such a pair exists,  we remove $(s_k,s_i)$ and
  $(s_j,s_i)$ from $\Delta$. Denote the remaining set as $\Delta'$. Then
  $\abs{\Delta'} \geq (n+1)n -2n$.

  According to Lemma~\ref{lm:t=2-difference1} and the definition of
  $\Delta'$, if there are two pairs in $\Delta'$ whose differences are
  equal, they must have the form $s_i-s_j=s_j-s_i$, and so,
  $2(s_i-s_j)=0$. Hence, for every $g \in G$ of order $2$, there are at most two pairs $(s,s') \in \Delta'$ with $s-s'=g$ and, for every other non-zero element of $G$, there is at most one such representation.  It follows that
  \[\abs{G}-1+m_2(G)\geq  (n+1)n -2n.\]
  Rearranging the terms, we complete the proof.
\qed\end{proof}

\begin{theorem}
  \label{th:no210} Let $n \geq 3$. Then 
  $\cB(n,2,1,0)$ lattice-tiles $\Z^n$ only when $n\in\{3, 5\}$, and
  only by $2$-splitting $\Z_7$ and $2$-splitting  $\F_2^4$,
  respectively.
\end{theorem}
\begin{proof}
  By Lemma~\ref{lem:order2}, if we are to have a splitting
  $G=\set{1}\splt_2 S$, then
  \[ \binom{n}{2}+n+1+{m_2(G)}\geq n^2-n+1,\]
  where we used the fact that $G=\set{1}\splt_2 S$ implies
  $\abs{G}=\abs{\cB(n,2,1,0)}$. Rearranging we get,
  \begin{equation}
    \label{eq:mgeq}
    m_2(G)\geq \frac{1}{2}n(n-3).
  \end{equation}

  We now turn to look at $G$. Since it is Abelian, we may write
  \[ G = \Z_{n_1}\times \Z_{n_2} \times \dots \times \Z_{n_\ell},\]
  with $n_1,\dots,n_\ell\geq 2$. We observe that
  \[ m_2(\Z_{n_i})=\begin{cases}
  0 & \text{$n_i$ is odd,} \\
  1 & \text{otherwise.}
  \end{cases}\]
  Thus,
  \[ m_2(G) = \parenv*{\prod_{i=1}^{\ell} (m_2(\Z_{n_i})+1)}-1.\]
  If $G\neq \F_2^r$, then necessarily
  \begin{equation}
    \label{eq:mleq}
    m_2(G)\leq \frac{1}{2}\abs{G}-1= \frac{1}{4}(n^2+n-2),
  \end{equation}
  which is attained by setting exactly one of the $n_i$ to be $4$, and
  the rest to be $2$.
  
  If we compare \eqref{eq:mgeq} and \eqref{eq:mleq}, then for $n\geq
  7$ the lower bound of \eqref{eq:mgeq} is greater than the upper
  bound of \eqref{eq:mleq}, hence, only $G=\F_2^r$ is still
  possible. For $n\leq 6$ we deal with the cases separately:
  \begin{itemize}
  \item
    For $n=3$, $\abs{G}=7$, hence $G=\Z_7$.   A splitting set $S=\{1,2,4\}$ can be found in  \cite[Theorem 6]{BuzEtz12}.  
   \item
    For $n=4$, $\abs{G}=11$, hence $G=\F_{11}$, but $m_2(G)=0$,
    contradicting \eqref{eq:mgeq}.
  \item
    For $n=5$, $\abs{G}=16$, with the following options:
    \begin{itemize}
    \item
      $G=\F_2^4$, for which $m_2(G)=15$, and a $2$-splitting exists by
      Theorem~\ref{th:perfectcode} (see Example~\ref{ex:ex2}).
    \item
      $G=\Z_4\times\F_2^2$, for which $m_2(G)=7$, but a computer
      search rules out such a splitting.
    \item
      $G=\Z_4^2$, for which $m_2(G)=3$, contradicting \eqref{eq:mgeq}.
    \item
      $G=\Z_8\times\F_2$, for which $m_2(G)=3$, contradicting \eqref{eq:mgeq}.
    \item
      $G=\Z_{16}$, for which $m_2(G)=1$, contradicting \eqref{eq:mgeq}.
    \end{itemize}
  \item
    For $n=6$, $\abs{G}=22$, hence $G=\F_2\times\F_{11}$, but
    $m_2(G)=1$, contradicting \eqref{eq:mgeq}.
  \end{itemize}
  Finally, if $n\geq 7$, only $G=\F_2^r$ remains an option, but by
  Theorem~\ref{th:revperfectcode} we must then have a perfect
  $[n,k,5]$ linear code over $\F_2$, and such codes do not exist
  (e.g., see~\cite{MacSlo78}).  \qed\end{proof}

  Using a similar method, we now direct our attention to the case of
  $\cB(n,2,2,0)$.  Let $G$ be an Abelian group and assume that $G \geq
  \{1,2\} \splt_2 S$, for some $S=\{s_1,s_2,\ldots, s_n\} \subseteq
  G$. Denote $s_{n+i}\eqdef 2s_i$ for $1\leq i\leq n$ and $s_\infty \eqdef 0$. 
  Consider the congruence modulo $n$. We assume that $\infty \equiv \infty \pmod{n}$, and $\infty \not \equiv i \pmod{n}$ and $i \not \equiv \infty \pmod{n}$ for all $i \in [1,2n]$.
  Let
  \[\Delta \eqdef \set*{(s_i,s_j) ;  i,j \in [1,2n]\cup \set{\infty}, i\not\equiv j \ppmod{n}}.\]
    Then $\abs{\Delta}= (2n+1)2n-2n=4n^2$. We are to estimate
  the number of the equations
  \[s_i-s_j=s_k-s_\ell,\]
  where $(s_i,s_j),(s_k,s_\ell) \in \Delta$ and $(i,j)\neq (k,\ell)$. Note
  that the equation implies
  \[s_i+s_\ell=s_k+s_j.\]
  Since $G \geq \set{1,2} \splt_2 S$, either $i\equiv \ell \pmod{n}$ or
  $k \equiv j \pmod{n}$.  By exchanging the two sides of the
  equations, we assume that $i\equiv \ell \pmod{n}$ always holds.

  \begin{lemma}\label{lm:t=2-kp=2-difference1}
    In the setting above, the number of the equations
    \[s_i-s_j=s_k-s_\ell,\]
    where $(s_i,s_j),(s_k,s_\ell) \in \Delta$, $i\equiv \ell \pmod{n}$ and
    $k\not \equiv j\pmod{n}$, is at most $8n$.
  \end{lemma}

  \begin{proof}
    If $i=\ell=\infty$, then $s_j+s_k=0$. Since $G \geq \set{1,2}
    \splt_2 S$, necessarily $j\equiv k \pmod{n}$, contradicting the
    assumption.
    
    Now, let $\bai$ be the unique integer of $[1,n]$ such that
    $\bai\equiv i \equiv \ell \pmod{n}$.
    \begin{enumerate}
    \item
      If $i=\ell=\bai$, then $2s_{\bai}=s_k+s_j$. Since $G \geq
      \set{1,2} \splt_2 S$, necessarily $(k,j)\in \set*{(\bai+n,\infty),(\infty,\bai+n)}$.
    \item
      If $i=\bai$ and $\ell=n+\bai$, then
      $s_{\bai}-s_j=s_k-2s_{\bai}$. We claim that there is at most one
      pair $\{j,k\}$ with $j\not \equiv k \pmod{n}$ such that the
      equality holds; otherwise, suppose we have another pair
      $\set{j',k'}$ satisfying the conditions, then
      $s_j+s_k=s_{j'}+s_{k'}$, contradicting the fact that $G \geq
      \set{1,2} \splt_2 S$.
    \item
      If $(i,\ell)=(\bai+n,\bai)$ or $(\bai+n,\bai+n)$, we have the
      same claim as that in case 2.
    \end{enumerate}
    According to the argument above, given $\bai \in [1,n]$, if
    $i\equiv \ell\equiv \bai \pmod{n}$, we have at most four pairs
    $\set{j,k}$ such that the equation holds. The conclusion follows
    since each pair can generate two equations.\qed
  \end{proof}

  Let $m_3(G)$ be the number of elements of order $3$ in $G$, i.e.,
  \[m_3(G)\eqdef  \abs*{\set*{x\in G ; x\neq 0, 3x=0}}.\]

  \begin{lemma}\label{lm:t=2-kp=2-difference2}
    In the setting above, further assume that the order of $G$ is odd.
    Then the number of the equations
    \[s_i-s_j=s_k-s_\ell,\]
    where $(s_i,s_j),(s_k,s_\ell) \in \Delta$, $i\equiv \ell \pmod{n}$ and
    $k \equiv j\pmod{n}$, is at most
    \[2m_3(G)+11n+11.\]
  \end{lemma}

  \begin{proof}
    Let $\bai,\baj\in[1,n]\cup\set{\infty}$ such that $\bai\equiv
    i\equiv \ell \pmod{n}$ and $\baj\equiv j \equiv k \pmod{n}$.  By the definition of $\Delta$, we have $i\not \equiv j \pmod{n}$, and so, $\bai\neq \baj$.
    The equation $s_i-s_j=s_k-s_\ell$ implies that
    \[as_{\bai}-bs_{\baj}=cs_{\baj}-ds_{\bai},\]
    for some $a,b,c,d\in \set{1,2}.$ We discuss the number of
    equations for each possible value of $(a,b,c,d)$.

    \begin{enumerate}
    \item
      If $a=b=c=d=1$, then $2s_{\bai}=2s_{\baj}$, contradicting $G
      \geq \{1,2\} \splt_2 S$.
    \item
      If $a+d=2$, then there are at most $n+1$ ordered pairs
      $(\bai,\baj)$ such that the equation holds; otherwise, by the
      pigeonhole principle there exist two ordered pairs
      $(\bai,\baj)$, and $(\bai',\baj)$ satisfying the equation, with
      $\bai\neq\bai'$. Then we get that
      $(a+d)s_{\bai}=(b+c)s_{\baj}=(a+d)s_{\bai'},$ i.e., $2s_{\bai}
      =2s_{\bai'}$ for some $\bai\not =\bai'$, a contradiction.
    \item
      If $a+d=4$, then again there are at most $n+1$ ordered pairs
      $(\bai,\baj)$ such that the equation holds; otherwise, we have
      $4s_{\bai} =4s_{\bai'}$ for some $\bai\not =\bai'$,
      contradicting the assumption that $\abs{G}$ is odd.
    \item
      If $b+c=2$ or $4$, we have the same claim as that in cases 2 and 3.
    \item
      If $(a,b,c,d)=(2,2,1,1)$ or $(1,1,2,2)$, then  
      \[2s_{\bai}-2s_{\baj}=s_{\baj}-s_{\bai}\]
      and 
      \[s_{\bai}-s_{\baj}=2s_{\baj}-2s_{\bai}.\]
      Rearranging the terms, we have $3(s_{\baj}-s_{\bai})=0$ and
      $3(s_{\bai}-s_{\baj})=0$. Thus the total number of such two
      kinds of equations is at most $m_3(G)$.
    \item  If $(a,b,c,d)=(2,1,2,1)$ or $(1,2,1,2)$, then 
      \[2s_{\bai}-s_{\baj}=2s_{\baj}-s_{\bai}\]
      and 
      \[s_{\bai}-2s_{\baj}=s_{\baj}-2s_{\bai}.\]
      If the equations above occur, then the equations in case 5 also
      occur. Thus the total number of such two kinds of equations is
      also at most $m_3(G)$.
    \end{enumerate}
    Note that cases 2,3 and 4 include $2^4-4-1=11$ possible values of
    $(a,b,c,d)$. The conclusion follows by summing up all the numbers
    discussed above.  \qed\end{proof}

    \begin{lemma}\label{lem:order3}
      Suppose that $G \geq \set{1,2} \splt_2 S$, and let $n \eqdef
      \abs{S}$. If $\abs{G}$ is odd, then we have that
      \[\abs{G} +2m_3(G) \geq 4n^2 - 19n-10.\]
    \end{lemma}
    \begin{proof}
      Combining Lemma~\ref{lm:t=2-kp=2-difference1} and
      Lemma~\ref{lm:t=2-kp=2-difference2}, we repeat the same
      arguments as in Lemma~\ref{lem:order2} to obtain the result.
    \qed\end{proof}

    We can now state and prove the result on $\cB(n,2,2,0)$.
    \begin{theorem}
      \label{th:no220}
      Let $n \geq 3$, then $\cB(n,2,2,0)$ lattice-tiles $\Z^n$ only
      when $n\in\{3,11\}$, and only by $2$-splitting $\Z_{19}$ and
      $2$-splitting $\F_3^{5}$, respectively.
    \end{theorem}
    \begin{proof}
      Note that $\abs{\cB(n,2,2,0)}=2n^2+1$, which is odd.  By
      Lemma~\ref{lem:order3}, if we are to have a splitting
      $G=\set{1,2}\splt_2 S$, then
      \[ 2n^2+1 +2m_3(G) \geq 4n^2 - 19n-10.\]
      Rearranging we get,
      \begin{equation}
        \label{eq:mgeq220}
        m_3(G)\geq \frac{1}{2}(2n^2-19n-11).
      \end{equation}

      We now turn to look at $G$.  Write
      \[ G = \Z_{n_1}\times \Z_{n_2} \times \dots \times \Z_{n_\ell},\]
      where $n_1,\dots,n_\ell\geq 3$, and all of them are odd.  Then
      \[ m_3(G) = \parenv*{\prod_{i=1}^{\ell} (m_3(\Z_{n_i})+1)}-1.\]
      Since 
      \[ m_3(\Z_{n_i})=\begin{cases}
      2 & \text{$n_i$ is divisible by $3$,} \\
      0 & \text{otherwise,}
      \end{cases}\]
      if $G\neq \F_3^r$, then necessarily
      \begin{equation}
        \label{eq:mleq220}
        m_3(G)\leq \frac{1}{3}\abs{G}-1= \frac{2n^2-2}{3},
      \end{equation}
      which is attained by setting exactly one of the $n_i$ to be $9$,
      and the rest to be $3$.
  
      If we compare \eqref{eq:mgeq220} and \eqref{eq:mleq220}, then
      for $n\geq 30$ the lower bound of \eqref{eq:mgeq220} is greater
      than the upper bound of \eqref{eq:mleq220}, hence, only
      $G=\F_3^r$ is still possible. However, if $G=\F_3^r$, by
      Theorem~\ref{th:revperfectcode} we must then have a perfect
      $[n,k,5]$ linear code over $\F_3$, and such codes do not exist
      if $n\neq 11$ (e.g., see~\cite{MacSlo78}).
  
      For $11\leq n \leq 29$, \eqref{eq:mgeq220} implies $m_3(G) \geq
      \frac{1}{2}(n(2n-19)-11)\geq 11.$ Necessarily $27$
      divides $\abs{G}$. The only two possible cases are $n=11$ with
      $\abs{G}=243$, and $n=16$ with $\abs{G}=513$.
  
      We deal with the remaining cases separately:
      \begin{itemize}
      \item
        For $n=3$, $\abs{G}=19$, hence $G=\Z_{19}$.  A splitting set
        $S=\{1,11,7\}$ can be found in \cite[Theorem 6]{BuzEtz12}.
      \item
        For $n=4$, the non-existence is shown in \cite[Corollary
          6]{BuzEtz12}.
      \item
        For $n\in\{5,6,8,9, 10\}$, $\abs{G}$ is square-free, hence $G$ is
        cyclic. A computer search rules out these cases.
      \item 
        For $n=7$, $\abs{G}=99$, hence $G=\Z_{9}\times \Z_{11}$ (which is isomorphic to $\Z_{99}$) or
        $\F_{3}\times \Z_{33}$. A computer search rules out these two
        cases.
      \item For $n=11$, $\abs{G}=243$, with the following options. 
        \begin{itemize}
        \item
          $G=\F_3^5$, for which $m_3(G)=242$.
        \item
          $G=\Z_9\times\F_3^3$, for which $m_3(G)=80$.
        \item 
          $G=\Z_9\times\Z_9 \times \F_3$, for which $m_3(G)=26$.
        \item 
          $G=\Z_{27} \times \F_3^2$, for which $m_3(G)=26$.
        \item 
          $G=\Z_{27} \times \Z_9$, for which $m_3(G)=8$, contradicting
          \eqref{eq:mgeq220}.
        \item 
          $G=\Z_{81} \times \F_3$, for which $m_3(G)=8$, contradicting
          \eqref{eq:mgeq220}.
        \item 
          $G=\Z_{243}$, for which $m_3(G)=2$,
          contradicting \eqref{eq:mgeq220}.
        \end{itemize}
        A computer search rules out the groups $\Z_9\times\F_3^3$,
        $\Z_9\times\Z_9 \times \F_3$ and $\Z_{27} \times \F_3^2$.
        When $G=\F_3^5$, a $2$-splitting exists by
        Theorem~\ref{th:perfectcode} (see Example~\ref{ex:ex4}).
      \item
        For $n=16$, $\abs{G}=513=27\times 19$, hence $m_3(G)\leq 26$,
        contradicting \eqref{eq:mgeq220}. \qed
      \end{itemize}  
    \end{proof}
  
\section{Conclusion}
\label{sec:conclude}

In this paper we studied general tilings as well as lattice tilings of
$\Z^n$ with $\BALL$. These may act as perfect error-correcting codes
over a channel with at most $t$ limited-magnitude errors. We
constructed such lattice tilings from perfect codes in the Hamming
metric, and provided several non-existence results. We summarize some
of our non-existence results for lattice tilings and below, where it is
interesting to note the difference between the cases of
$\frac{t}{n}<\frac{1}{2}$ and $\frac{t}{n}\geq \frac{1}{2}$.

\begin{corollary}
Let $2\leq t < n/2$, and $\kp\geq \km\geq 0$ not both $0$.  Then
$\BALL$ cannot lattice-tile $\Z^n$ when one of the following holds:
\begin{enumerate}
\item $\frac{(\km+1)^2}{\kp+\km+1} \geq \binom{n}{t}^{1/t}$.
\item $t <n/4$, $\km=0$ and $\kp\geq  2 \binom{n}{t}-2$.
\item $t=2$, $\km=0$, $\kp=1$ and  $n\neq 5$.
\item  $t=2$, $\km=0$, $\kp=2$ and $n\neq 11$.
\end{enumerate}
\end{corollary}

\begin{corollary}
Let  $2\leq  t < n\leq 2t$, and $\kp\geq \km\geq 0$ not both $0$. If $\BALL$ lattice-tiles $\Z^n$, then one of the following holds:
\begin{enumerate}
\item $\km=0$ and one of the following holds:
 \begin{enumerate}
  \item $t=n-1$(such tilings  have been constructed in \cite{Ste90,BuzEtz12});
  \item $(2n-2)/3 \leq t \leq n-3$ and $\kp=1$\footnote{Recall that the entire case of $t=n-2$ has been excluded in  \cite{BuzEtz12}.};
  \item $n/2\leq t<(2n-2)/3$;
 \end{enumerate}
\item $\kp=\km$ and one of the following holds:
 \begin{enumerate}
  \item $(4n-2)/5 \leq t \leq n-1$ and $\kp=\km=1$;
  \item $n/2 \leq t < (4n-2)/5$ and $\sum_{i=1}^t{n\choose i} (2\kp)^{i-1} \geq (\kp+1)^{t}$.
 \end{enumerate}
\end{enumerate}
\end{corollary}

It is also interesting to compare the results here, when $t\geq 2$,
with the known results for $t=1$. The non-existence results we have
here rely heavily on geometric arguments, or general algebraic
arguments. The notable exceptions are Theorem~\ref{th:no210} and
Theorem~\ref{th:no220}, which carefully study the structure of the
group being split. This is in contrast with the strong non-existence
results when $t=1$, due to the fact that when $t=1$, if $G$ is split
then so is the cyclic group of the same size, $\Z_{\abs{G}}$. This
does not hold when $t\geq 2$, as evident, for example, during the
proof of Theorem~\ref{th:no210}, where $\F_2^4$ is $2$-split but
$\Z_{16}$ is not.

Whether some strong statement may be said about the structure of the
group being split, remains as an open question for further
research. It is also interesting to ask whether more $t$-splittings
exist, namely, whether $t$-splittings exist which are not derived from
perfect codes in the Hamming metric. Finally, it remains open whether
any other non-lattice tilings of $\BALL$ exist.

\end{document}